\documentclass{amsart}
\usepackage{amsmath, amssymb, amscd, amsbsy, amsthm}
\newtheorem{theorem}{Theorem}[section]
\newtheorem{lemma}[theorem]{Lemma}
\newtheorem{prop}[theorem]{Proposition}
\newtheorem{cor}[theorem]{Corollary}
\theoremstyle{definition}
\newtheorem{definition}[theorem]{Definition}

\theoremstyle{remark}

\numberwithin{equation}{section}

\def\Q{{\mathbb Q}}

\def\ad{{\operatorname{ad}}}
\def\CC{{\mathcal{C}}}
\def\EQ{{\mathcal{E}_{\mathbb Q}}}
\def\EV{{\mathcal{E}_{V}}}
\def\L{{\Lambda}}
\def\lcd{{{\sf lcd}}}
\def\lcdb{{\overline{\sf lcd}}}
\def\maru{{\nabla}}
\def\PP{{\mathcal{P}}}
\def\Par{{\underbar{\sf Par}}}
\def\Tree{{\underbar{\sf Tree}}}
\begin{document}
\title{The Lie algebra of rooted planar trees}
\author{Tomohiko Ishida}
\address{Graduate School of Mathematical Sciences \\
University of Tokyo, 3-8-1 Komaba \\
Meguro-ku, Tokyo 153-8914, Japan.}
\email{ishidat@ms.u-tokyo.ac.jp}
\author{Nariya Kawazumi}
\address{Department of Mathematical Sciences \\
University of Tokyo, 3-8-1 Komaba \\
Meguro-ku, Tokyo 153-8914, Japan.}
\email{kawazumi@ms.u-tokyo.ac.jp}
\date{\today}

\begin{abstract}
We study a natural Lie algebra structure on the free vector 
space generated by all rooted planar trees as the associated 
Lie algebra of the nonsymmetric operad (non-$\Sigma$ operad, 
preoperad) of rooted planar trees.  
We determine whether the Lie algebra and some related Lie algebras 
are finitely generated or not, and prove that 
a natural surjection called the augmentation homomorphism 
onto the Lie algebra of polynomial vector fields on the line 
has no splitting preserving the units.
\end{abstract}

\maketitle

\section{Introduction}

The Lie algebra of polynomial vector fields on the line, 
$W_1 = \Q[x]\frac{d}{dx}$, and its Lie subalgebras 
$L_0 = x\Q[x]\frac{d}{dx}$ and $L_1 = x^2\Q[x]\frac{d}{dx}$
have been studied in the context of Gel'fand-Fuks theory. 
In particular, Goncharova \cite{G} computed the cohomology group
$H^\ast (L_1)$ completely. Based on her monumental work, various 
studies including \cite{FF} \cite{Mi} and \cite{V} have been 
developed. See also \cite{Is}, \cite{K1} and \cite{K2}. 
On the other hand, Kuno and the second author \cite{KK} discovered 
a Lie algebra structure on the free $\Q$-vector space generated 
by the set of all linear chord diagrams, $\mathcal{LC}$,
and a surjective homomorphism $\kappa: \mathcal{LC} \to L_0$. 
The Lie algebra $\mathcal{LC}$ is purely combinatorial and 
comes from the derivation Lie algebra of the tensor algebra 
of a symplectic vector space. So it seems to have no 
relation with Gel'fand-Fuks theory. \par

The link between the linear chord diagrams and the vector fields 
on the line is the notion of a nonsymmetric operad, or equivalently 
a non-$\Sigma$ operad or a preoperad. 
Kapranov and Manin \cite{KM} introduced 
a Lie algebra $\L(\PP)$ associated to a nonsymmetric operad of $\Q$-vector 
spaces $\PP$. To understand the homomorphism $\kappa$, we introduce 
the augmentation homomorphism of the Lie algebra induced from 
a nonsymmetric operad of sets. We denote $\PP = \Q\CC$, if
$\PP((m))$, $m \geq 0$, is the free $\Q$-vector space of 
$\CC((m))$ for a nonsymmetric operad of sets $\CC$. 
The augmentation maps
$\Q\CC((m)) \to \Q$ induce a natural homomorphism of Lie algebras
$\varepsilon: \L(\Q\CC) \to W_1$, which we call {\it the augmentation 
homomorphism}. The Lie algebra $\mathcal{LC}$ is regarded as 
the Lie algebra induced from an operad of sets, 
and the homomorphism 
$\kappa: \mathcal{LC} \to L_0$ is derived from the augmentation 
homomorphism. \par

In this paper we study two fundamental problems 
for some nonsymmetric operad of sets $\CC$;
\begin{enumerate}
\item[(i)] Is the Lie algebra $\L(\Q\CC)$ finitely generated?
\item[(ii)] Does the augmentation homomorphism have a splitting preserving 
the units $1 \in \CC((1))$ and $x\frac{d}{dx} \in L_0$?
\end{enumerate}

As typical examples of nonsymmetric operads of sets, 
we have the nonsymmetric operad of rooted planar trees $\Tree$
and its nonsymmetric suboperad of binary planar trees $\Tree_2$. 
We prove both of the questions for both of the nonsymmetric operads 
have negative answers 
(Theorems \ref{gen-tree}, \ref{tree-infinite} and \ref{main}).
In order to prove Theorem \ref{main}, we introduce 
the nonsymmetric operad of partitions 
$\Par$ and its nonsymmetric suboperad $\Par_2$ of binary partitions. 
The answers of (i) and (ii) for $\Par$ and that of (ii) for $\Par_2$ 
are negative
(Theorems \ref{partition-infinite} and \ref{main}), 
while that of (i) for $\Par_2$ 
is affirmative (Theorem \ref{gen-par}). Here it should be remarked 
Loday and Ronco \cite{LR} have already studied algebraic structures 
on binary rooted planar trees in a different way from ours. 
The answer of the question (i) for the Lie algebra $\mathcal{LC}$ 
is negative \cite{KK}, while that of (ii) is still open. \par

In this paper we work over the rationals $\Q$, but all the results 
hold true over any field of characteristic zero. 
An operad without assuming the symmetric 
group action has various names; a non-$\Sigma$ operad \cite{MSS}, 
a preoperad \cite{KP}, an asymmetric operad, and a nonsymmetric 
operad \cite{LV}. For details, see \cite{KP}.  
As will be shown in this paper, the notion of an operad 
without assuming the symmetric group action is quite fundamental.
In this paper we adopt {\it a nonsymmetric operad} 
following \cite{LV}. 
\par

\vskip 5pt

\noindent \textbf{Acknowledgments.}
The authors wish to express their gratitude to
Yusuke Kuno for valuable discussions, and 
to Jean-Louis Loday for careful comments to the first version
of this paper. The second-named author also thanks 
Ralph Kaufmann and Robert Penner for helpful advices. 
The first-named
author is supported by JSPS Research Fellowships
for Young Scientists (23$\cdot$1352).
The second-named author is partially supported 
by the Grant-in-Aid for Scientific Research (A) (No.20244003) 
from the Japan Society for Promotion of Sciences.

\section{The Lie algebra associated to a nonsymmetric operad}

We begin by recalling the definition of a {\it nonsymmetric operad} or 
a {\it non-$\Sigma$ operad} or a {\it preoperad} of $\mathbb{Q}$-vector
spaces. 
\begin{definition}
A sequence of $\Q$-vector spaces $\PP = \{\PP((m))\}_{m\geq 0}$ is 
a nonsymmetric operad of $\Q$-vector spaces, if it admits an element 
$1 \in \PP((1))$ called {\it the unit} and $\Q$-linear maps called 
{\it the composition maps}
$$
\gamma = \gamma^\PP: \PP((k))\otimes\PP((j_1))\otimes\cdots\otimes\PP((j_k)) 
\to \PP((\sum^k_{s=1}j_s)), \quad k \geq 1, j_s \geq 0, 
$$
which satisfy the following two conditions. 
\begin{enumerate}
\item {\rm (Associativity)} For any $c \in \PP((k))$, $d_s \in \PP((j_s))$, 
$1 \leq s \leq k$, and $e_t \in \PP((i_t))$, $1 \leq t \leq j = 
\sum^k_{s=1}j_s$, we have 
$$
\gamma(\gamma(c\otimes d_1\otimes\cdots\otimes d_k)\otimes 
e_1\otimes\cdots\otimes e_j) 
= \gamma(c\otimes f_1\otimes\cdots\otimes f_k),
$$
where $f_s = \gamma(d_s\otimes e_{j_1+\cdots+j_{s-1}+1}\otimes\cdots\otimes
e_{j_1+\cdots+j_{s-1}+j_s})$. 
\item {\rm (Unit)} We have $\gamma(1\otimes d) = d$ and 
$\gamma(c\otimes 1^{\otimes k}) = c$ for any $d \in \PP((j))$ and 
$c \in \PP((k))$, 
$k \geq 1$. 
\end{enumerate}
\end{definition}

As usual, we denote
$$
c\circ_sd_s := \gamma(c\otimes 1^{\otimes(s-1)}\otimes 
d_s\otimes 1^{\otimes (k-s)})
\in \PP((k+j_s-1)), \quad 1 \leq s \leq k.
$$
\par

A nonsymmetric operad of sets $\CC$ is defined in a similar way. 
For any $c \in \CC((k))$, $d_s \in \CC((j_s))$, 
$1 \leq s \leq k$, we denote the composition by 
$\gamma(c; d_1,\dots, d_s) \in \CC((\sum^k_{s=1}j_s))$. 
Then we denote by 
$\Q\CC$ the nonsymmetric operad of $\Q$-vector spaces defined by 
$$
(\Q\CC)((m)) := \Q(\CC((m))),
$$
the free $\Q$-vector space generated by the set $\CC((m))$, $m \geq 0$. \par

For any $\Q$-vector space $V$, the endomorphism operad $\EV$ is defined by 
$$
\EV(m) := \operatorname{Hom}(V^{\otimes m}, V)
$$
with the obvious unit and composition maps. The augmentation maps of the 
free $\Q$-vector space $\Q\CC((m))$, $\varepsilon: \Q\CC((m)) \to \Q = 
\operatorname{Hom}(\Q^{\otimes m}, \Q) = \EQ(m)$, 
$\sum_{x \in \CC((m))}a_xx \mapsto 
\sum_{x \in \CC((m))}a_x$, define a homomorphism of 
nonsymmetric operads of $\Q$-vector spaces
$$
\varepsilon: \Q\CC \to \EQ,
$$
which we call the augmentation homomorphism. \par

Kapranov and Manin \cite{KM} define two Lie algebras associated to
an operad of $\Q$-vector spaces. One requires the symmetric 
group action, but the other denoted by 
$$
\L(\PP) := \bigoplus^\infty_{m=0}\PP((m))
$$
can be defined for any nonsymmetric operad of $\Q$-vector spaces, $\PP$. 
See also \cite{LV} 5.3.16 and 5.8.17. 
The Lie bracket $[c,d]$, $c \in \PP((k))$, $d \in \PP((j))$, is 
defined by 
$$
[c,d] := \sum^j_{t=1}d\circ_tc - \sum^k_{s=1}c\circ_sd \in \PP((k+j-1)).
$$
Here it should be remarked our sign convention is different from that 
in \cite{KM}, in order to make the bijection $\L(\EQ)\overset\cong\to W_1 
:= \Q[x]\frac{d}{dx}$ stated below an isomorphism of Lie algebras. \par
To check the Jacobi identity of $\L(\PP)$, we write simply
$$
c(d) := \sum^j_{t=1}d\circ_tc.
$$
Then the map
$$
\delta: \L(\PP) \to \operatorname{End}(\L(\PP)), \quad 
c \mapsto (\delta_c: d \mapsto c(d))
$$
is injective since $\delta_c(1) = c$. 
One computes $\delta_{[c,d]} = [\delta_c, \delta_d] \in 
\operatorname{End}(\L(\PP))$. $\L(\PP)$ inherits the Jacobi 
identity from the Lie algebra $\operatorname{End}(\L(\PP))$
by the injection $\delta$. Here we remark the Lie algebra $\L(\PP)$ 
has a finer structure, a {\it pre-Lie algebra}. 
For details, see \cite{KM} 1.7 and \cite{LV} 5.8.17.
\par

For a finite dimensional $\Q$-vector space $V$, we have a natural 
isomorphism of Lie algebras onto the derivation Lie algebra of $T(V^*)$
\begin{equation}
\L(\EV) = \operatorname{Der}(T(V^*)),
\label{2der}
\end{equation}
where $T(V^*) = \bigoplus^\infty_{m=0}(V^*)^{\otimes m}$ is the 
tensor algebra of the dual space $V^* = \operatorname{Hom}(V, \Q)$. 
In order to describe the isomorphism (\ref{2der}) explicitly 
for the case $V = \Q$, we denote the element corresponding to 
$1 \in \Q = \operatorname{Hom}(\Q^{\otimes m}, \Q)$ by $1_m 
\in \EQ(m)$, $m \geq 0$. Then we have 
$$
[1_m, 1_n] = (n-m)1_{m+n-1}.
$$
This means the map given by 
$$
1_m \in \L(\EQ) \mapsto x^m\frac{d}{dx} \in W_1
$$
is an isomorphism onto the Lie algebra of polynomial vector 
fields on the line, $W_1 = \Q[x]\frac{d}{dx}$, 
For the rest of this paper we identify $\L(\EQ) = W_1$ 
through this isomorphism. \par

Thus, for any nonsymmetric operad of sets $\CC$, the augmentation 
homomorphism $\varepsilon: \Q\CC \to \EQ$ induces 
a natural homomorphism of Lie algebras
$$
\varepsilon: \L(\Q\CC) \to \L(\EQ) = W_1,
$$
which we call also {\it the augmentation homomorphism}. 
If $\CC((m)) \neq \emptyset$ for each $m \geq 0$, 
it is surjective. It is natural to ask 
whether it does split or not. The answer to 
this question should describe the complexity 
of the given nonsymmetric operad $\CC$. \par

As usual, we denote $L_k := x^{k+1}\Q[x]\frac{d}{dx}$ 
for any $k \geq -1$, which is a Lie subalgebra of $W_1$. 
Similarly we denote 
$$
\L_k(\PP) := \bigoplus^\infty_{m=k+1}\PP((m)),
$$
which is also a Lie subalgebra of $\L(\PP)$. 
The augmentation homomorphism induces 
a homomorphism of Lie algebras
$$
\varepsilon:\L_k(\Q\CC) \to L_k
$$
for each $k \geq -1$. \par

We denote by $e_0 = {e_0}^\PP \in \L(\PP)$ the unit 
$1 \in \PP((1))$ regarded as an element of the Lie algebra
$\L(\PP)$. When $\PP = \EQ$, we have ${e_0}^\EQ = x\frac{d}{dx}
\in W_1$. For any $m \geq 0$, the subspace $\PP((m)) \subset
\L(\PP)$ is exactly the $(m-1)$-eigenspace of the adjoint 
action of the unit, $\ad e_0$. Hence, for any nonsymmetric operads
of $\Q$-vector spaces $\PP$ and $\PP'$, if a homomorphism 
of Lie algebras $\varphi: \L(\PP) \to \L(\PP')$, 
which is not necessarily the induced homomorphism of a 
homomorphism of nonsymmetric operads, preserves the units 
$\varphi({e_0}^\PP) = {e_0}^{\PP'}$, then we have 
$\varphi(\PP((m))) \subset \PP'((m))$ for any $m \geq 0$.\par
In view of the action of $e_0$ we find out the center
$Z(\L(\PP))$ satisfies
\begin{equation}
Z(\L(\PP)) \subset Z(\L_0(\PP)) \subset Z(\PP((1))).
\end{equation}

Here we regard $\PP((1))$ as a Lie subalgebra of $\L(\PP)$. 
The standard chain complex $C_*(\L(\PP))$ of the Lie algebra 
$\L_k(\PP)$, $k \geq -1$, is decomposed into the eigenspaces
of the adjoint action $\ad e_0$. The $l$-eigenspace of $\ad e_0$, 
$C_*(\L_k(\PP))_{(l)}$ is a subcomplex of 
$C_*(\L_k(\PP))$. We denote 
$$
H_*(\L_k(\PP))_{(l)} := H_*(C_*(\L_k(\PP))_{(l)}). 
$$
Clearly we have 
$$
H_*(\L_k(\PP)) = \bigoplus^\infty_{l=k}H_*(\L_k(\PP))_{(l)}.
$$
The formula $\ad e_0 = d\circ(e_0\wedge) + (e_0\wedge)\circ d$ 
on the standard chain complex implies 
$$
H_*(\L_k(\PP)) = H_*(\L_k(\PP))_{(0)}
$$
for $k = -1$ or $0$. In particular, if a nonsymmetric operad of sets 
$\CC$ satisfies the condition $\sharp\CC((0)) = \sharp\CC((1)) 
= \sharp\CC((2)) = 1$, then we have $C_*(\L(\Q\CC))_{(0)}
= C_*(W_1)_{(0)} = C_*(sl_2(\Q))_{(0)}$, so that 
\begin{equation}
H_*(\L(\Q\CC)) = H_*(W_1) = H_*(sl_2(\Q)) = 
\begin{cases}
\Q, & \mbox{if $* = 0, 3$,}\\
0, & \mbox{otherwise.}
\end{cases}
\label{2sl_2}
\end{equation}
Similarly, if $\sharp\CC((1)) = 1$, then 
\begin{equation}
H_*(\L_0(\Q\CC)) = H_*(L_0) = 
\begin{cases}
\Q, & \mbox{if $* = 0, 1$,}\\
0, & \mbox{otherwise.}
\end{cases}
\label{2u_1}
\end{equation}
\par

We conclude this section by a comment on the Lie algebra of 
linear chord diagrams, $\mathcal{LC}$, introduced by 
Kuno and the second author \cite{KK}. 
We denote the free $\Q$-vector space generated by the set of 
linear chord diagrams of $m$ chords, $m \geq 1$, by $\lcd(2m-1)$, 
while we define $\lcd(2m) = 0$. The $j$-th amalgamation of two 
linear chord diagrams $C$ and $C'$, $C\ast_jC'$, defined in 
\cite{KK}, gives a composition map on $\lcd = \{\lcd(n)\}_{n \geq 0}$. 
In a similar way to \cite{KK}, we can prove $\lcd$ is an anticyclic 
operad. The Lie algebra $\mathcal{LC}$ is exactly $\L(\lcd)$. 
What we have stated in this section is a straight-forward 
generalization of some of observations in \cite{KK}. 
As a nonsymmetric operad, we have $\lcd = \Q\lcdb$, 
where $\lcdb$ is the operad of sets consisting of all linear 
chord diagrams. The homomorphism $\kappa: \mathcal{LC} \to L_0$ 
in \cite{KK} is the composite of the augmentation homomorphism and 
the homomorphism
$$
x\Q[x^2]\frac{d}{dx} \to L_0 = x\Q[x]\frac{d}{dx}, \quad
x^{2n+1}\frac{d}{dx} \mapsto 2x^{n+1}\frac{d}{dx}.
$$
For an anticyclic operad $\PP$, we denote by $\L^+(\PP)$ 
the cyclic invariants in $\L(\PP)$. 
One can prove $\L^+(\PP)$ is a Lie subalgebra of $\L(\PP)$. 
If $\PP=\lcd$, the Lie algebra $\L^+(\PP)$ is exactly 
the Lie algebra of (circular) chord diagrams $\mathcal{C}$
introduced in \cite{KK}.

\section{The nonsymmetric operad of rooted planar trees}

We recall the definition of the nonsymmetric operad of rooted planar trees, 
$\Tree$, following Markl, Shnider and Stasheff \cite{MSS} I.1.5.
Let $\Tree((m))$ be the set of planar trees with $1$ root at the 
bottom and $m$ leaves at the top, regarded as labeled from left 
to right; $1$ through $m$. For $S \in \Tree((m))$, $T \in \Tree((n))$ 
and $1 \leq i \leq m$, $S\circ_iT$ is defined to be the tree obtained by 
grafting the root of $T$ to the $i$-th leaf of $S$. This operation 
makes the sequence $\Tree := \{\Tree((m))\}_{m \geq 1}$ 
a nonsymmetric operad of 
sets, which we call {\it the nonsymmetric operad of rooted planar trees}. 
It is known $\Tree$ is a free nonsymmetric operad. See \cite{LV} 5.8.6
and \cite{MSS} II.1.9. \par
For $n \geq 2$, we denote by $\Tree_n((m))$ the subset of $\Tree((m))$ 
consisting of trees all of whose vertices are of valency $\leq n+1$. 
The sequence $\Tree_n := \{\Tree_n((m))\}_{m \geq 1}$ is a nonsymmetric 
suboperad of $\Tree$. We call it {\it the nonsymmetric operad of 
$n$-ary rooted planar trees}. 
As is known, each element of the set $\Tree((m))$ corresponds 
to a meaningful way of inserting one set of parentheses into the 
word $12\cdots m$, that is, a cell of the Stasheff associahedron 
$K_m$ \cite{S}. For example, $\Tree((1)) = \{1\}$, $\Tree((2)) = \{(12)\}$, 
and $\Tree((3)) = \{((12)3), (1(23)), (123)\}$. The set $\Tree_2((m))$ 
corresponds exactly to the vertices of the associahedron $K_m$.
From (\ref{2u_1}) we have 
$$
H_*(\L(\Q\Tree)) = H_*(\L(\Q\Tree_n)) = H_*(L_0). 
$$ 
\par

Now we introduce an enhancement of the nonsymmetric operad $\Tree$. 
To do this, we consider the $i$-th face $\partial_ic$ of $c \in 
\Tree((m))$ for $m \geq 2$ and $1\leq i \leq m$, defined by 
erasing the $i$-th leaf of $c$. For example, $\partial_i((12)3)
= \partial_i(1(23)) = (12)$, $\partial_i((1(23))4) = ((12)3)$ 
for $1 \leq i\leq 3$, and $\partial_4((1(23))4) = (1(23))$. 
Let $\Tree^-((0))$ be a singleton, whose unique element we denote 
by $\maru$. We define $\Tree^-((m)) := \Tree((m))$ for $m \geq 1$, 
$\partial_11 := \maru$, and $c\circ_i\maru := \partial_ic$ 
for $c \in \Tree((m))$, $m \geq 1$ and $1 \leq i \leq m$. 
Then $\Tree^-:= \{\Tree^-((m))\}_{m \geq 0}$ forms a nonsymmetric operad 
of sets. For $n\geq 2$, the sequence $\Tree^-_n 
:= \{\Tree_n^-((m))\}_{m \geq 0}$, 
given by $\Tree^-_n((0)) = \Tree^-((0))$ and $\Tree^-_n((m)) = \Tree_n((m))$ 
for $m\geq 1$, is a nonsymmetric suboperad of $\Tree^-$. 
From (\ref{2sl_2}) we have 
$$
H_*(\L(\Q\Tree^-)) = H_*(\L(\Q\Tree^-_n)) = H_*(W_1). 
$$ 
Clearly we have $\partial_i\partial_jc= \partial_{j-1}\partial_ic$ 
if $i < j$. Hence the linear map
$$
\partial:= \sum^m_{i=1}\partial_i: 
\Q\Tree^-((m)) \to \Q\Tree^-((m-1))
$$
satisfies $\partial\partial = 0$, so that $\Q\Tree^-((*)) = 
\{\Q\Tree^-((m)), \partial\}_{m\geq 0}$ is a chain complex, and 
$\Q\Tree^-_n((*)) = \{\Q\Tree^-_n((m)), \partial\}_{m\geq 0}$ 
a subcomplex. Consider the tree $(12) \in \Tree^-_2((2))$. 
Then we have $\partial((12)\circ_2c) = c - (12)\circ_2\partial c$
for any $c \in \Tree^-((m))$, $m \geq 0$. This implies the vanishing
of the homology groups
$$
H_*(\Q\Tree^-((*))) = H_*(\Q\Tree^-_n((*))) = 0.
$$

\section{The nonsymmetric operad of partitions}

In this section we introduce the nonsymmetric operad of partitions, 
$\Par$. \par

Let $\Q[x_i; i \geq 1]$ be the rational polynomial ring in infinitely 
many indeterminates $\{x_i\}_{i \geq 1}$. A monomial 
$\prod^N_{i=1}{x_i}^{a_i}$ with $N \geq 2$, $a_i \geq 1$ 
($1 \leq \forall i \leq N$), 
corresponds to the nontrivial order-preserving partition of the set 
$\{1,2,\dots, m\}$ with $m= \sum^N_{i=1}a_i$ given by 
$$
\{1,2,\dots, m\} = \coprod^N_{i=1}\{\sum^{i-1}_{j=1}a_j + 1, 
\sum^{i-1}_{j=1}a_j + 2, 
\dots, \sum^{i-1}_{j=1}a_j + a_i\}.
$$
For $m\geq 2$, we define 
$$
\Par((m)) := \left\{\prod^N_{i=1}{x_i}^{a_i}; N \geq 2, a_i \geq 1 
(1 \leq \forall i \leq N),
\,\mbox{and}\, \sum a_i = m\right\},
$$
which is regarded as the set of nontrivial order-preserving partitions 
of the set $\{1,2,\dots, m\}$. The composition map is defined by 
$$
\gamma\left(\prod^N_{i=1}{x_i}^{a_i}; \prod^{N_1}_{k=1}{x_k}^{a_{1k}}, 
\dots, \prod^{N_m}_{k=1}{x_k}^{a_{mk}}\right)
:= \prod^N_{i=1}{x_i}^{b_i},
$$
where $$
b_i = \sum^{a_1+\cdots+a_{i-1}+a_i}_{j=a_1+\cdots+a_{i-1}+1}
\left(\sum^{N_j}_{k=1}a_{jk}\right).
$$
In other words, we define 
$$
\left(\prod^N_{i=1}{x_i}^{a_i}\right)\circ_s
\left(\prod^{N_s}_{k=1}{x_k}^{a_{sk}}\right)
:= {x_l}^{a_l-1+\sum^{N_s}_{k=1}a_{sk}}\prod_{i\neq l}{x_i}^{a_i}
$$
if $a_1+\cdots+a_{l-1}+1 \leq s \leq a_1+\cdots+a_{l-1}+a_l$. 
It is easy to check this composition satisfies the axiom of 
associativity. But there does not exist a unit in the polynomial 
ring $\Q[x_i; i \geq 1]$. In fact, $\left(\prod^N_{i=1}{x_i}^{a_i}\right)
\circ_jx_1 = \prod^N_{i=1}{x_i}^{a_i}$, but
$x_1\circ_1\left(\prod^N_{i=1}{x_i}^{a_i}\right) = {x_1}^{\sum a_i}$.  
So we define $\Par((1))$ to be a singleton, 
whose unique element we denote by $1$, and 
$$
\left(\prod^N_{i=1}{x_i}^{a_i}\right)\circ_j1 =
1\circ_1\left(\prod^N_{i=1}{x_i}^{a_i}\right) := \prod^N_{i=1}{x_i}^{a_i}.
$$
Then $\Par := \{\Par((m))\}_{m \geq 1}$ forms a nonsymmetric operad of sets, 
which we call {\it the nonsymmetric operad of partitions}. 
For $n \geq 2$, we denote $\Par_n((1)) := \Par((1)) = \{1\}$ and 
$$
\Par_n((m)) := \Par((m))\cap  \Q[x_1, x_2,\dots, x_n]
$$
for $m \geq 2$. Then $\Par_n := \{\Par_n((m))\}_{m \geq 1}$ 
is a nonsymmetric suboperad of $\Par$. We call it 
{\it the nonsymmetric operad of $n$-ary partitions}. \par

The reason why we introduce the nonsymmetric operad $\Par$ is to simplify 
the Lie algebra $\L(\Q\Tree)$ by using the following homomorphism 
$\nu: \Tree\to \Par$. \par
Let $c$ be a rooted planar tree in $\Par((m))$, $m \geq 2$. 
Look at the nearest vertex to the root. 
Each edge except the one attached to the root has the set of leaves 
sitting above itself. Hence the tree $c$ gives a nontrivial order-preserving 
partition of the set of leaves $\{1,2,\dots, m\}$, which we denote by 
$\nu(c) \in \Par((m))$. For example, $\nu((1(23))4) = {x_1}^3x_2$, 
$\nu((12)(34)) = {x_1}^2{x_2}^2$. Further we define $\nu(1) := 1 \in
\Par((1))$.  Then the maps $\nu: \Tree((m)) \to \Par((m))$, $m \geq 1$, form a
homomorphism of nonsymmetric operads
$$
\nu: \Tree \to \Par
$$
from the definition of the composition maps in $\Par$. Clearly it induces 
a homomorphism of nonsymmetric operads
$$
\nu: \Tree_n \to \Par_n
$$
for each $n \geq 2$. \par

To compute the Lie bracket on $\L_1(\Q\Par)$, we regard the polynomial ring 
$\Q[x_i; i \geq 1]$ as an $L_0$-module by the diagonal action. 
More precisely, $\xi(x)\frac{d}{dx} \in L_0$ acts on $f(x_1, x_2, \dots) 
\in \Q[x_i; i \geq 1]$ by 
$$
\left(\xi(x)\frac{d}{dx}\right)(f(x_1, x_2, \dots))
= \sum^\infty_{i=1}\xi(x_i)\frac{\partial}{\partial x_i}f(x_1, x_2, \dots).
$$

Then it is easy to prove the following.
\begin{lemma}\label{3bracket}
For $c, d \in \L_1(\Q\Par) \subset \Q[x_i; i \geq 1]$ we have 
$$
[c,d] = \varepsilon(c)(d) - \varepsilon(d)(c) 
\in \L_1(\Q\Par) \subset \Q[x_i; i \geq 1].
$$ 
\end{lemma}
As a corollary, we obtain
\begin{cor}\label{3abel}
The kernel of the augmentation homomorphism
$\varepsilon: \L(\Q\Par) =\L_0(\Q\Par) \to \L_0(\EQ) = L_0$ is abelian. 
\end{cor}
The Lie bracket on $\L_1(\Q\Par)$ extends to the Laurent polynomial ring 
in infinitely many indeterminates $\Q[{x_i}^{\pm1}; i \geq 1]$, and 
makes it a Lie algebra.


\section{The Lie algebra $\Lambda (\Q\Tree _2^-)$ is not finitely generated}
In this section, we prove the following theorem.
\begin{theorem}\label{gen-tree}
The Lie algebra $\Lambda(\Q\Tree _2^- )$ is not finitely generated. 
\end{theorem}

As a preliminary of the proof of Theorem \ref{gen-tree}, 
we show Lemma \ref{catalan} and Lemma \ref{h_m}.


\begin{lemma}\label{catalan}
For any $m\geq 2$, 
$$
H_1( \Lambda _1(\Q\Tree _2))_{(m)}\neq 0.
$$
\end{lemma}
\begin{proof}
The cardinality of $\Tree _2((m+1))$ is the $m$-th Catalan number
\[ c(m)=\frac{1}{m+1}\binom{2m}{m}, \]
and it coincides with the dimension of $C_1(\Lambda _1\Tree _2)_{(m)}$.
Let $c'(m)$ denote the dimension of the second chain complex 
$C_2(\Lambda _1\Tree _2)_{(m)}$.
\par

We prove that $c'(m)<c(m)$ for any $m\geq 1$.
Since $c'(m)$ can be computed from the equation
\[ c'(m)=\begin{cases}\displaystyle\sum _{l=1}^k c(l)c(m-l) & 
\text{(if }m=2k+1\text{)} \\
\displaystyle\sum _{l=1}^{k-1} c(l)c(m-l)+\binom{c(k)}{2} & 
\text{(if }m=2k\text{)} \end{cases}, \]
we have the inequality
\[ c'(m)\leq \frac{1}{2}\sum _{l=1}^{m-1} c(l)c(m-l). \]
By the well-known recurrence equation
\[ c(m)=\sum _{l=0}^{m-1} c(l)c(m-l-1), \]
we have 
\begin{align*}
c'(m)+c(m)\leq \frac{1}{2}\sum _{l=0}^m c(l)c(m-l) =\frac{1}{2}c(m+1).
\end{align*}
Since the ratio of consecutive Catalan numbers is described as
\[ \frac{c(m+1)}{c(m)}=\frac{2(2m+1)}{m+2}, \]
we obtain finally
\begin{align*}
c'(m)\leq \frac{m-1}{m+2}c(m) <c(m). 
\end{align*}
Therefore, $H_1(\L_1(\Q\Tree_2))_{(m)}$ does not vanish for any $m\geq 1$.
\end{proof}

\begin{cor}
The Lie subalgebra $\Lambda _1(\Q\Tree _2)$ of $\Lambda(\Q\Tree _2^-)$ 
is not finitely generated. 
\end{cor}
To prove Theorem \ref{gen-tree}, 
we define ${\mathfrak h}_m$ to be the Lie subalgebra of 
$\Lambda(\Q\Tree _2^-)$  generated by $\bigcup_{j=2}^m \Tree_2((j))$.
\begin{lemma}\label{h_m}
The vector subspace $\Q\Tree _2^-((0))\oplus\Q\Tree _2^-((1))
\oplus{\mathfrak h}_m$ is a Lie subalgebra of $\Lambda (\Q\Tree_2^-)$.
\end{lemma}
\begin{proof}
It is obvious that the subspace $\Q\Tree _2^-((0))\oplus\Q\Tree_2^-((1))$ 
is a Lie subalgebra of $\Lambda (\Q\Tree _2^1)$.
Hence it is sufficient to prove that the inclusions
$$ 
(\ad 1)({\mathfrak h}_m), (\ad\maru )({\mathfrak h}_m)\subset \Q\Tree
_2^-((1))\oplus{\mathfrak h}_m 
$$
hold.\par
Now we consider an arbitrary Lie algebra ${\mathfrak g}$.
For any $n$ elements $u_1, u_2, \dots , u_n$ of ${\mathfrak g}$ 
and a binary tree $c$ which belongs to $\Tree _2((n))$,
we define $f_c(u_1, u_2, \dots , u_n)$ to be an element 
of ${\mathfrak g}$ obtained from $u_1, u_2, \dots , u_n$ 
by the Lie bracket 
following the parentheses corresponding to $c$. 
For example, $f_{(1)}(u_1)=u_1, f_{((12)(34))}(u_1, u_2, u_3, u_4)=
[[u_1, u_2], [u_3, u_4]]$.
By the Jacobi's identity, the equation
\begin{equation}\label{f_c}
(\ad v)f_c(u_1, u_2, \dots , u_n)=\sum _{i=1}^nf_c(u_1, u_2,  
\dots , u_{i-1}, (\ad v)u_i, u_{i+1}, \dots , u_n)
\end{equation}
holds for any $v\in{\mathfrak g}$.
\par

Under these settings, the Lie algebra ${\mathfrak h}_m$ is 
the vector subspace spanned by the set 
$$
\{ f_c(u_1, u_2, \dots , u_n); n\geq 1, 
c\in \Tree _2((n)), u_i\in\Tree _2((j_i)), 2\leq j_i\leq m\}.
$$
Hence the assertion holds if we prove 
that $f_c(u_1, u_2,  \dots , u_{i-1}, (\ad v)u_i, u_{i+1}, \dots , u_n)$ 
is in $\Q\Tree _2((1))\oplus{\mathfrak h}_m$ for any $u_i$'s and 
$v=1$ and $\maru$. 
\par
Since $(\ad 1)(u)=ju$ for any $u\in \Tree_2((j))$, 
the claim holds true for $v=1$.
\par

In the case $n=1$, $f_c(u_1)=u_1$ and $(\ad\maru)(u_1)$ 
is in $\mathbb{Q}\Tree _2((j_1-1))$.
In the case $n\geq 2$, if $j_i\geq 3$, then $(\ad\maru)(u_i)$ 
is in $\mathbb{Q}\Tree _2((j_i-1))$.
If $j_i=2$, then $u_i=(12)$ and $(\ad\maru)((12))=2\cdot 1$ and
\begin{align*}
f_c(u_1, u_2, \dots , u_{i-1}, (\ad\maru)(12), u_{i+1}, \dots , u_n)
&=2f_c(u_1, u_2, \dots, u_{i-1}, 1, u_{i+1}, \dots , u_n) \\
&=Cf_{\partial _ic}(u_1, u_2, \dots , u_{i-1}, u_{i+1}, \dots , u_n)
\end{align*}
for some integer $C$.
Here $\partial_ic$ is the $i$-th face of $c$ defined in \S3. 
The claim holds true also for $v=\maru$.
\end{proof}
\begin{proof}[Proof of Theorem \ref{gen-tree}]
Assume that $\Lambda(\Q\Tree_2^-)$ is finitely generated. 
Then there exists a sufficiently large $m\geq 2$ so that 
$\Lambda(\Q\Tree _2^-)$ is generated by $\bigoplus_{j=0}^m \Q\Tree_2((j))$.
In other words, $\Lambda(\Q\Tree_2^-)$ has the decomposition
$$ 
\Lambda(\Q\Tree_2^-)=\Q\Tree_2((0))\oplus \Q\Tree_2((1))\oplus{\mathfrak
h}_m.
$$
In particular, if $l>m$, the inclusion
$$ 
{\mathfrak h}_m\cap \Q\Tree_2((l))\subset [\Lambda_1(\Q\Tree_2^-), 
\Lambda_1(\Q\Tree_2^-)] 
$$
holds.
This implies that $H_1(\Lambda_1(\Q\Tree_2))_{(l-1)}=0$ and it
contradicts Lemma \ref{catalan}. This concludes the proof of Theorem
\ref{gen-tree}.
\end{proof}
\section{The Lie algebra $\Lambda(\Q\Par_2)$ is finitely generated}

In this section, we prove the following theorem.
\begin{theorem}\label{gen-par}
The Lie algebra $\Lambda_1(\Q\Par_2)$ is generated 
by $x_1x_2, x_x^2x_2, x_1x_2^2$, and
$x_1^3x_2+x_1x_2^3$. 
\end{theorem}
\begin{proof}
It should be remarked $\dim C_1(\L(\Q\Par_2))_{(m-1)} 
= \dim \Q\Par_2((m)) = m-1$. 
It is obvious that $x_1x_2, x_x^2x_2$ and $x_1x_2^2$ 
are not in the derived ideal 
$[\L_1(\Q\Par_2),\L_1(\Q\Par_2)]$.
Next we compute brackets which take values in $\Q\Par_2((4))$. 
We obtain
\begin{align*}
&[x_1x_2, x_1^2x_2]=x_1^3x_2+x_1^2x_2^2-x_1x_2^3, \\
&[x_1x_2, x_1x_2^2]=-x_1^3x_2+x_1^2x_2^2+x_1x_2^3. 
\end{align*}
Hence $x_1^3x_2+x_1x_2^3$ is not in the derived ideal.
\par

On the other hand, if $m\geq 5$, any elements of $\Q\Par_2((m-1))$ 
can be obtained by
repetition of Lie brackets. In fact, if $m=5$, then
\begin{align*}
&[x_1x_2, x_1^3x_2]=2x_1^4x_2+x_1^3x_2^2-x_1x_2^4, \\
&[x_1x_2, x_1^2x_2^2]=-x_1^4x_2+2x_1^3x_2^2+2x_1^2x_2^3-x_1x_2^4, \\
&[x_1x_2, x_1x_2^3]=-x_1^4x_2+x_1^2x_2^3+2x_1x_2^4, \\
&[x_1^2x_2, x_1x_2^2]=-2x_1^4x_2+x_1^3x_2^2-x_1^2x_2^3+2x_1x_2^4,
\end{align*}
and thus the boundary map
$$
\delta _{1, 5}\colon C_2(\Lambda (\Par _2))_{(4)}\to 
C_1(\Lambda (\Par_2))_{(4)}
$$
can be represented by the matrix
$$ 
A_5 =\begin{pmatrix}
2 & 1 & 0 & -1 \\
-1 & 2 & 2 & -1 \\
-1 & 0 & 1 & 2 \\
-2 & 1 & 1 & 2 \\
\end{pmatrix} .
$$

Hence $\det A_5\neq 0$.
Further if $m\geq 6$, then
\begin{align*}
&[x_1x_2, x_1^{m-k-1}x_2^k]
=-x_1^{m-1}x_2+(m-k-1)x_1^{m-k}x_2^k+kx_1^{m-k-1}x_2^{k+2}-x_1x_2^{m-1} \\
&\qquad\qquad\qquad\qquad\qquad\qquad\qquad\qquad\qquad\qquad\qquad\qquad 
\text{(for any }2\leq k\leq m-3\text{)}, \\ 
&[x_1x_2, x_1x_2^{m-2}]=-x_1^{m-1}x_2+x_1^2x_2^{m-2}+(m-3)x_1x_2^{m-1}, \\
&[x_1^2x_2, x_1x_2^{m-3}]
=-2x_1^{m-1}x_2+x_1^3x_2^{m-3}-x_1^2x_2^{m-2}+(m-3)x_1x_2^{m-1}, \\
&[x_1^2x_2, x_1^2x_2^{m-4}]
=-2x_1^{m-1}x_2+2x_1^4x_2^{m-4}+(m-5)x_1^2x_2^{m-2} 
\end{align*}
and thus a matrix representation $A_m$ of the boundary map
$$
\delta _{1, m}\colon C_2(\Lambda (\Par_2))_{(m-1)}\to 
C_1(\Lambda (\Par_2))_{(m-1)}
$$
has the $(m-1)\times (m-1)$ submatrix
$$
A'_m=\begin{pmatrix}
-1 & m-3 & 2 & & & & & -1 \\
-1 &  & m-4 & 3 & & & O & -1 \\
\vdots & & & \ddots & \ddots & & & \vdots \\
-1 & & & & 3 & m-4 &  & -1 \\
-1 & O & & & & 2 & m-3 & -1 \\
-1 & &  & & & & 1 & m-3 \\
-2 & 0 & \dots & 0 & 0 & 1 & -1 & m-3 \\
-2 & 0 & \dots & 0 & 2 & 0 & m-5 & 0 \\
\end{pmatrix}.
$$
Hence 
\begin{align*}
\det A'_m&=\frac{1}{6}(-1)^{m}(m-3)! \cdot \det
\begin{pmatrix}
-1 & 3 & m-4 & 0 & -1 \\
-1 & 0 & 2 & m-3 & -1 \\
-1 & 0 & 0 & 1 & m-3 \\
-2 & 0 & 1 & -1 & m-3 \\
-2 & 2 & 0 & m-5 & 0 \\
\end{pmatrix} \\
&=\frac{1}{6}(-1)^{m+1}(m-3)!\cdot m(m-1)(2m-7) \\
&\neq 0.
\end{align*}
Consequently, the boundary map $\delta_{1, m}$ is surjective 
if $m\geq 5$ and this concludes the
proof of Theorem \ref{gen-par}.
\end{proof}

\begin{cor}
$$ 
H_1(\Lambda_1 (\Q\Par_2) ;\Q)\cong\Q^4.
$$
\end{cor}

\begin{cor}
The Lie algebra $\Lambda(\Q\Par_2)$ is generated by $1$, 
$x_1x_2, x_x^2x_2$, $x_1x_2^2$, and $x_1^3x_2+x_1x_2^3$.
\end{cor}


In contrast, the following proposition holds.
\begin{prop}\label{partition-infinite}
The Lie algebra $\Lambda(\Q\Par)$ is not finitely generated. 
\end{prop}
\begin{proof}
Assume that $\Lambda(\Q\Par)$ is finitely generated.
Then there exists a sufficiently large $m\geq 2$ 
so that $\Lambda(\Q\Par)$ is generated by $\bigoplus_{j=0}^m \Q\Par((j))$.
However, the Lie subalgebra $\Lambda (\Q\Par_m)$ of $\Lambda (\Q\Par)$ 
contains $\bigoplus_{j=0}^m \Q\Par ((j))$
although it doesn't generate $\Lambda(\Q\Par)$.
Thus we have a contradiction.
\end{proof}
Since $\nu\colon\Tree\to\Par$ induces a surjective homomorphism 
of Lie algebras, we directly have the first half of the following corollary.
The rest is proved by an argument similar to the proof of 
Theorem \ref{gen-tree}.

\begin{cor}\label{tree-infinite}
The Lie algebra $\Lambda(\Q\Tree)$ is not finitely generated.
Furthermore, neither the Lie algebra $\Lambda(\Q\Tree^-)$ is.
\end{cor}


\section{The augmentation homomorphism on $\Lambda (\Q\Tree)$ 
has no splitting}

Let $\varepsilon$ and $\varepsilon_1$ denote the augmentation homomorphism 
from $\Lambda
(\Q\Par)$ and $\Lambda (\Q\Tree )$ to $L_0$, respectively.  
In this section, we prove the following theorem.
\begin{theorem}\label{main}
The augmentation homomorphism $\varepsilon _1\colon 
\Lambda (\Q\Tree )\to L_0$ 
has no splitting preserving the units.
\end{theorem}
In the proof of Theorem \ref{main}, the nonsymmetric operad $\Par$ plays 
an important role.
We denote by $\nu: \L(\Q\Tree) \to \L(\Q\Par)$ the homomorphism of 
Lie algebras induced by $\nu: \Tree \to \Par$. 
Then we have $\varepsilon_1=\varepsilon\circ\nu: 
\L(\Q\Tree) \to \L(\Q\Par) \to L_0$. Hence it suffices to prove that
$\varepsilon\colon\Lambda (\Q\Par)\to L_0$ 
has no splitting preserving the units. 
If such a splitting would exist, it must map $\Q x^m\frac{d}{dx}$ to 
$\Q\Par((m))$ for each $m \geq 2$. In fact, both of them 
are the $(m-1)$-eigenspaces 
of $\ad {e_0}^{\EQ}$ and $\ad {e_0}^{\Q\Par}$, respectively. 
\par

Let $\iota\colon\Par\to\Par$ be the involution defined by 
$$ 
\iota (x_1^{a_1}x_2^{a_2}\dots x_n^{a_n})
=x_1^{a_n}x_2^{a_{n-1}}\dots x_n^{a_1}. 
$$
Then it is obvious that $\iota$ induces an automorphism 
of the Lie algebra $\Lambda _1(\Q\Par)$.
If we denote by $\Lambda _1(\Q\Par)^\pm$ the ($\pm1$)-eigenspace 
of the involution 
$$ 
\Lambda_1(\Q\Par)^\pm=\{u\in\Lambda_1(\Q\Par)^\pm;\iota (u)=\pm u \}, 
$$
then $\Lambda_1(\Q\Par)^+$ is a Lie subalgebra and 
$[\Lambda_1(\Q\Par)^+, \Lambda_1(\Q\Par)^-] 
\subset \Lambda_1(\Q\Par)^-$. 
Since the kernel of the augmentation homomorphism $\varepsilon$ 
includes $\Lambda_1(\Q\Par)^-$, we have 
$[\Lambda_1(\Q\Par)^-, \Lambda_1(\Q\Par)^-] = 0$. 
Hence $\Lambda_1(\Q\Par)$ is the semi-direct product of 
$\Lambda_1(\Q\Par)^+$ and $\Lambda_1(\Q\Par)^-$
\begin{equation}
\Lambda_1(\Q\Par) = \Lambda_1(\Q\Par)^-\rtimes \Lambda_1(\Q\Par)^+.
\label{semi-direct}
\end{equation}
Since $\varepsilon(\Lambda_1(\Q\Par)^-) = 0$, we have a factorization
$$
\varepsilon\mid_{\Lambda _1(\Q\Par)} =\varepsilon _2\circ p: 
\Lambda _1(\Q\Par)\overset{p}\to\Lambda_1(\Q\Par)^+
\overset{\varepsilon_2}\to L_0,
$$
where $p$ is the second projection in (\ref{semi-direct}) 
and $\varepsilon_2$ is the restriction of the augmentation homomorphism to
$\Lambda_1(\Q\Par)^+$.
Therefore, in order to establish Theorem \ref{main}, it suffices to 
prove the following proposition.
\begin{prop}
The augmentation homomorphism 
$\varepsilon _2\colon \Lambda _1(\Q\Par) ^+\to L_1$ 
has no splitting 
which maps $\Q x^{m}\frac{d}{dx}$ to $\Q\Par((m))$ for each $m\geq 2$.
\end{prop}
\begin{proof}
We assume that there exists a splitting $s$ which maps $\Q x^{m}\frac{d}{dx}$ 
to $\Q\Par((m))$ for each $m\geq 2$.
We denote
$$
e_i=x^{i+1}\frac{d}{dx} \in L_1
$$
for $i \geq 1$. 
Recall that $L_1$ is generated by $e_1$ and $e_2$.
In fact, $e_n$ is obtained from $e_1$ and $e_2$ by
\[ e_n =\frac{1}{(n-2)!}(\ad e_1)^{n-2}(e_2), \]
for $n\geq 3$.
Therefore the splitting $s$ is uniquely determined 
by its values of $e_1$ and $e_2$.
Since $\L_1(\Q\Par)^+\cap \Q\Par((2))$ is generated by $x_1x_2$ and 
$\L_1(\Q\Par)^+\cap \Q\Par((3))$ by $x_1^2x_2+x_1x_2^2$ and $x_1x_2x_3$, 
the value $u_1$ of $e_1$ by $s$ must be $x_1x_2$ and $u_2$ of $e_2$ 
must have the form
\[ u_2=\frac{t}{2}(x_1^2x_2+x_1x_2^2)+(1-t)x_1x_2x_3. \]
If we define $u_n$ by
\[ u_n=\frac{1}{(n-2)!}(\ad u_1)^{n-2}(u_2) \]
for $n\geq 3$,
then the equation $u_n=s(e_n)$ must hold also for $n\geq 3$.
In particular, $u_5$ must coincide with $[u_2, u_3]$ since $e_5=[e_2, e_3]$.
To prove that $u_5\neq [u_2, u_3]$,  
we compute $u_3$, $u_4$, and $u_5$ explicitly.
Then we obtain
\begin{align*}
u_3&=[x_1x_2, \frac{t}{2}(x_1^2x_2+x_1x_2^2)+(1-t)x_1x_2x_3] \\
&=tx_1^2x_2^2-(1-t)(x_1^3x_2+x_1x_2^3)+(1-t)x_1x_2^2x_3
+(1-t)(x_1^2x_2x_3+x_1x_2x_3^2), \\
u_4&=\frac{1}{2}[x_1x_2, u_3] \\
&=\frac{-1+3t}{2}(x_1^3x_2^2+x_1^2x_2^3)
+\frac{-4+3t}{2}(x_1^4x_2+x_1x_2^4) \\
&\quad +(1-t)\{ x_1x_2^3x_3+x_1^2x_2x_3^2
+(x_1^2x_2^2x_3+x_1x_2^2x_3^2)+(x_1^3x_2x_3+x_1x_2x_3^3)\} ,
\end{align*}
and
\begin{align*}
u_5&=\frac{1}{6}[x_1x_2, u_4] \\
&=(2t-3)x_1^3x_2^3
+\left( 2t-\frac{7}{6}\right)
(x_1^4x_2^2+x_1^2x_2^4)+(2t-3)(x_1^5x_2+x_1x_2^5) \\
&\quad +\frac{1-t}{3}\{ (x_1x_2^3x_3^2+x_1^2x_2^3x_3)
+3x_1x_2^4x_3+3x_1^2x_2^2x_3^2+2(x_1^2x_2^3x_3+x_1x_2^3x_3^2) \\
&\qquad\qquad +3(x_1^3x_2x_3^2+x_1^2x_2x_3^3)
+3(x_1^3x_2^2x_3+x_1x_2^2x_3^3)+3(x_1^4x_2x_3+x_1x_2x_3^4)\} .
\end{align*}
On the other hand, 
\begin{align*}
[u_2, u_3]&=[\frac{t}{2}(x_1^2x_2+x_1x_2^2)+(1-t)x_1x_2x_3, u_3] \\
&=-2(1-t)x_1^3x_2^3+\frac{5}{2}t(1-t)(x_1^4x_2+x_1^2x_2^4)
+(t^2+t-3)(x_1^5x_2+x_1x_2^5) \\
&\quad +(1-t)\{ x_1x_2^4x_3+(x_1^2x_2x_3^3+x_1^3x_2x_3^2)
+(x_1^2x_2^3x_3+x_1x_2^3x_3^2) \\
&\qquad\qquad +(x_1^3x_2^2x_3+x_1x_2^2x_3^3)+(x_1^4x_2x_3+x_1x_2x_3^4)\} .
\end{align*}
Thus we have $u_5\neq [u_2, u_3]$, 
which contradicts $s$ is a homomorphism of Lie algebras. 
This completes the proof. 
\end{proof}

\bibliographystyle{amsplain}

\end{document}